\documentclass[12pt]{amsart}
\usepackage[utf8]{inputenc}
\usepackage[margin=1in]{geometry}
\usepackage[greek,english]{babel}
\usepackage{amsmath}
\usepackage{amsfonts}
\usepackage{amsthm}
\usepackage{amssymb}
\usepackage{mathtools}
\usepackage{tikz-cd}
\usepackage{enumerate}
\usepackage{mathabx}
\usepackage{enumitem}
\usepackage[pagebackref,colorlinks,linkcolor=red,citecolor=blue,urlcolor=blue,hypertexnames=true]{hyperref}
\usepackage[pagebackref,hypertexnames=true]{hyperref}
\numberwithin{equation}{section}

\newcommand{\Q}{\mathbb{Q}}
\newcommand{\Z}{\mathbb{Z}}
\newcommand{\E}{\mathcal{E}}
\newcommand{\C}{\mathbb{C}}

\newcommand{\N}{\mathbb{N}}

\newcommand{\TT}{\mathbb{T}}
\newcommand{\CC}{\mathcal{C}}

\newcommand{\la}{\lambda}
\newcommand\starop[2]{{
  \displaystyle\operatornamewithlimits{\raisebox{-0.5ex}[1.5ex][0ex]{\rm*}}_{#1}^{#2}}}

\makeatletter
\newcommand{\setword}[2]{%
  \phantomsection
  #1\def\@currentlabel{\unexpanded{#1}}\label{#2}%
}
\makeatother

\makeatletter
\newcommand\nnfootnote[1]{%
  \begin{NoHyper}
  \renewcommand\thefootnote{}\footnote{#1}%
  \addtocounter{footnote}{-1}%
  \end{NoHyper}
}

\makeatother

\newtheorem{theorem}{Theorem}[section]

\newtheorem{definition}[theorem]{Definition}
\newtheorem{lemma}[theorem]{Lemma}
\newtheorem{proposition}[theorem]{Proposition}

\newtheorem{example}[theorem]{Example}
\newtheorem{corollary}[theorem]{Corollary}
\newtheorem{conjecture}[theorem]{Conjecture}

\title{When amenable groups have real rank zero $C^*$-algebras}
\author{Iason Moutzouris}

\begin{document}
\maketitle
\nnfootnote{2020 Mathematics Subject Classification: 22D25, 46L05. \par 
}

\begin{abstract}
   We investigate when discrete, amenable groups have $C^*$-algebras of real rank zero. While it is known that this happens when the group is locally finite, the converse in an open problem. We show that if $C^*(G)$ has real rank zero, then all normal subgroups of $G$ that are elementary amenable and have finite Hirsch length must be locally finite.
\end{abstract}
\makeatletter
\@setabstract
\makeatother

\section{Introduction}
Let $G$ be a countable, discrete group. If $G$ is torsion free and amenable, then the Kadison-Kaplansky conjecture holds \cite[Thm. 1.3]{Baum-Connes_book}, so $C^*(G)\cong C^*_r(G)$ has no non-trivial projections. However, if $G$ is not torsion free, then for every $g\in G$ with finite order $n$, the element $\frac{1+g+g^2+...+g^{n-1}}{n}\in C^*(G)$ is a projection. If, in addition the group is locally finite, then $C^*(G)$ is an AF-algebra. \par
In \cite{effros}, Effros conjectured that if $C_r^*(G)$ is an AF-algebra, then $G$ has to be locally finite. The conjecture is still open. Actually, for amenable groups, we can strengthen the conjecture by replacing the AF-algebra assumption, with the weaker property of real rank zero:
\begin{conjecture}\label{conjecture}
    Let $G$ be a discrete, amenable group such that $C^*(G)$ has real rank zero. Then $G$ is locally finite.
\end{conjecture}
In the non-amenable case the situation is more complicated. For instance,
in \cite{Dykema_Rordam}, Dykema and R{\o}rdam proved that $C^*_r(\starop{n=2}\infty(\Z/n\Z))$ has real rank zero. Note that the free product $\starop{n=2}\infty(\Z/n\Z)$ is a non-amenable and non-periodic group.\par 
Conjecture \ref{conjecture} is open (even in the non-amenable case for full group $C^*$-algebras) but it has been verified for several classes of groups, including the ones below:
\begin{itemize}
    \item Discrete, amenable groups that have a subgroup of finite index that is torsion free. This is a direct Corollary of the proof of the Kadison-Kaplansky conjecture for amenable groups. For the sake of completion, we present the argument in Proposition \ref{torsion_free_by_finite}.
    \item Discrete, nilpotent groups. It was proved by Kaniuth in \cite{Kaniuth}. (Actually, this result is more general and provides a characterization for when (non-necessarily discrete), locally compact nilpotent groups have group $C^*$-algebras that have real rank zero. But, because on this paper we will talk only about discrete groups, we restrict on that special case).
    \item Finitely generated, discrete, elementary amenable groups. This was proved by Scarparo in \cite{scarparo}
\end{itemize}
The purpose of this paper is to come up with more obstructions that prevent $C^*(G)$ from having real rank zero. Our main result is the following;
\begin{theorem}\label{main-theorem}
    Let $G$ be a discrete, amenable group and assume that $C^*(G)$ has real rank zero. Then every normal subgroup of $G$ that is elementary amenable and has finite Hirsch length, has to be locally finite.
    \end{theorem}
Recall that because real rank zero is preserved under taking quotients, we can weaken the assumption of the Theorem above, by assuming that $H$ is a normal subgroup of some quotient of $G$. \par 
The Hirsch length was defined for elementary amenable groups by Hillman in \cite{hillman_finite_Hirsch_length_old}. We will explain more in Section 6, but for now we will only say that virtually polycyclic groups have finite Hirsch length. So, the following Corollary is automatic.
\begin{corollary}
Let $G$ be a discrete, amenable group and assume that $C^*(G)$ has real rank zero. Then every normal subgroup of $G$ that is virtually polycyclic has to be finite.
\end{corollary}
Theorem \ref{main-theorem} yields new examples of groups verifying Conjecture \ref{conjecture}. For instance, notice that if $G$ is amenable and has a normal subgroup isomorphic to $\Z^n$, then $C^*(G)$ does not have real rank zero. Hence, we can start with an amenable group $H$ that is infinitely generated (or not elementary amenable). Every group homomorphism $\sigma: H\rightarrow Aut(\Z^n)\cong GL(n,\Z)$ gives rise to a semidirect product $G=\Z^n\rtimes H$. Notice that $C^*(G)$ does not have real rank zero and $G$ is infinitely generated (or not elementary amenable). \par
For the proof of Theorem \ref{main-theorem}, we generalize ideas from \cite{scarparo} and \cite{real_rank_zero_an_eigenvalue_variation}. A key step is Proposition \ref{locally_finite_by_nice_is_well_behaved}. For its proof, we need to work with continuous fields. More specifically, we define a quantity for elements of continuous fields, called oscillation (see beginning of Section 2). This quantity is zero for self-adjoint elements of finite spectrum (Lemma \ref{lemma_cts_field_2}) and it cannot increase much under small perturbations (Lemma \ref{lemma_cts_field_1}). This strategy allows us to find self-adjoint elements on a group $C^*$-algebra that are far away from elements of finite spectrum. After proving Proposition \ref{locally_finite_by_nice_is_well_behaved}, we use group theory tools, like Tits Alternative \cite{tits_alternative} and a characterization of groups with finite Hirsch length from \cite{Hillman_finite_Hirsch_length}, to show our main result. \par 
The main difficulty when trying to extend our results even further (e.g dropping the assumption of finite Hirsch length on Theorem \ref{main-theorem}, or generalize Scarparo's result to all elementary amenable groups), arises from the fact that we can find increasing sequences of $C^*$-algebras with real rank greater than zero (e.g matrices over $C(\TT)$), such that the inductive limit has real rank zero. Actually, whether the inductive limit has real rank zero heavily depends on the structure of the connecting maps.\par 
To bypass this issue, in Definition \ref{excellent_group} we define a property for groups, whose presence implies that the full group $C^*$-algebra is not of real rank zero. A crucial feature of this property is that it is preserved when taking increasing unions. This will allow us to show that locally nilpotent groups whose group $C^*$-algebra has real rank zero, are locally finite (see Proposition \ref{locally_nilpotent_is_excellent}). \par 

All groups will be discrete, unless clearly stated otherwise. We will denote with $e$ the identity of a group. For $g\in G$ we define $\setword{\beta}{symbolref}(g)$  $:=1-\frac{g+g^{-1}}{2}\in C^*(G)$. For a topological space $X$, we denote with $\CC_X$ the set of connected components of $X$. The center (of a $C^*$-algebra, or a group) will be denoted by $Z(.)$, while the set of projections of a $C^*$-algebra, with $P(.)$.

\section{Continuous fields of \texorpdfstring{$C^*$}{C*}-algebras and real rank zero.}

Let $X$ be a compact, Hausdorff topological space. A unital $C^*$-algebra $A$, is a \emph{$C(X)$-algebra} if there exists
a unital *-homomorphism $\theta:C(X)\rightarrow Z(A).$ For $f\in C(X)$ and $a\in A$, we write $fa$ for $\theta(f)a$. Let $x\in X$. Due to the $C(X)$-structure on $A$, we can consider the \emph{fiber} $A(x)\cong A/C_0(X/{x})A.$ Thus we can consider the \emph{evaluation to the fiber} $\pi_x: A\rightarrow A(x)$. We say that $A$ is a \emph{(unital) continuous field}, if for every $a\in A, x\mapsto ||\pi_x(a)||$ is continuous. Moreover, by \cite[Lemma 2.1]{Dadarlat_cts_fields_f.d},
\begin{equation}
    \label{(1)}
    ||a||=\max\{||\pi_x(a)||, \hspace{3mm} x\in X\}
\end{equation}
Let $a\in A$. We define the \emph{oscillation} of $a$ to be 
$$\omega_A(a):=\displaystyle{\sup_{Y\in \CC_X} \sup\{\Big|||\pi_x(a)||-||\pi_y(a)||\Big|, \hspace{3mm} x,y\in Y\}},$$
where  $\CC_X$ is the set of connected components of $X$. When it is clear on which continuous field we consider the oscillation, we write $\omega(a)$ instead of $\omega_A(a)$. \par
Note that if $\dim(X)=0$, then all connected components are points, so $\omega(a)=0$ for every $a\in A$. \par 
We will prove a few basic Lemmas regarding the oscillation.
\begin{lemma}\label{lemma_cts_field_1}
Let $a,b\in A$. Then $|\omega(a)-\omega(b)|\leq 2||a-b||$.
\begin{proof}
By the triangle inequality, we have that for every $x,y\in X$, $$\Bigg| \Big| ||\pi_x(a)||-||\pi_y(a)|| \Big|- \Big| ||\pi_x(b)||-||\pi_y(b)|| \Big| \Bigg|\leq $$
$$\leq ||\pi_x(a)-\pi_x(b)||+||\pi_y(a)-\pi_y(b)||\leq 2||a-b||.$$
Hence $$|\omega(a)-\omega(b)|=$$ $$\Bigg| \displaystyle{\sup_{Y\in \CC_X} \sup\{\Big|||\pi_x(a)||-||\pi_y(a)||\Big|, \hspace{3mm} x,y\in Y\}}-\displaystyle{\sup_{Y\in \CC_X} \sup\{\Big|||\pi_x(b)||-||\pi_y(b)||\Big|, \hspace{3mm} x,y\in Y\}}\Bigg|\leq $$
$$\displaystyle{\sup_{Y\in \CC_X} \sup \{\Bigg| \Big| ||\pi_x(a)||-||\pi_y(a)|| \Big|- \Big| ||\pi_x(b)||-||\pi_y(b)|| \Big| \Bigg|, \hspace{3mm} x,y\in Y\}}\leq 2||a-b||.$$
\end{proof}
\end{lemma}

\begin{lemma}\label{lemma_cts_field_2}
Let $a$ be a self-adjoint element of $A$ with finite spectrum. Then $\omega(a)=0.$
\begin{proof}
Observe that $x\mapsto ||\pi_x(a)||$ is continuous and takes finitely many values. So, it has to be constant on each connected component. Thus $\omega(a)=0.$
\end{proof}
\end{lemma}
Combining the two aforementioned Lemmas we get the following Corollary.
\begin{corollary}\label{cor_cts_field_3}
Let $a=a^*\in A$ and assume that there exists a self-adjoint element with finite spectrum $b\in A$ such that $||a-b||\leq \frac{\varepsilon}{2}.$ Then $\omega(a)\leq \varepsilon.$

\end{corollary}
Let $X$ be compact, Hausdorff and $n\in \N$. Observe that $A:=C(X)\otimes M_n(\C)\cong C(X,M_n(\C))$ is a (trivial) continuous field. Each fiber is isomorphic to $M_n(\C)$ and $\pi_x: A\rightarrow M_n(\C)$ is the evaluation on $x$. \par 
Moreover, by identifying $a$ with $diag(a,a,...,a)$, we can view $A$ as a $C^*$-subalgebra of $M_n(A)$. Notice that $Z(A)=Z(M_n(A))$. It is not difficult to see that if $A$ is continuous field of $C^*$-algebras over $X$ with fibers $A(x)$, then $M_n(A)$ is a continuous field of $C^*$-algebras over $X$ with fibers $M_n(A(x)).$ \par
\begin{lemma}\label{lemma_cts_field_4}
    Let $A$ be a continuous field over a compact, Hausdorff space $X$. For every $n\in \N$ and $f\in M_n(C(X))$, we have $\omega_{M_n(A)}(f)=\omega_{M_n(C(X))}(f)$.
    \begin{proof}
       Fix $n\in \N$ and $x\in X$. Let $\pi_x^{(n)}: M_n(A)\rightarrow M_n(A)/C_0(X/\{x\})M_n(A)$ be the evaluation. Then $\pi_x^{(n)}(f)=f(x)1_{M_n(A)}$. Thus $||\pi_x^{(n)}(f)||=||f(x)||_{M_n(\C)}.$ Result follows.
    \end{proof}
\end{lemma}
\begin{example}\label{ex_1}
Consider the function $z:\TT\rightarrow \C$ such that $z(x)=x$ for every $x\in \TT$. $z$ is a unitary in $C(\TT)$ and actually it is $C(\TT)=C^*(z)$. Let $v=diag(z^{\la_1},...,z^{\la_k}) \in M_k(C(\TT))$, where $\la_i \in \Z$ for every $i$ and $\la_1 \neq 0$. Consider the self-adjoint element $\beta=1-\frac{(v+v^{-1})}{2}$. Its oscillation is $\omega(\beta)=2.$  Indeed, $\beta=diag(1-Re(z^{\la_1}),...,1-Re(z^{\la_k})).$ Notice that $\beta(1)=0$. Consider $x\in \TT$ such that $Re(x^{\la_1})=-1.$ Then $s(x)=diag(2,l_2,..l_k)$ for some real numbers $l_2,...,l_k$. So, this matrix has 2 as an eigenvalue. So $||\beta(x)||\geq 2$. Note that $||\beta(y)||\leq 2$ for every $y\in \TT$. Hence $||\beta(x)||=2$ and $\omega(\beta)=2.$
\end{example}

Let $A$ be a $C^*$-algebra. The notion of real rank was introduced by L.Brown and Pedersen in \cite{real_rank_zero_intro}. A $C^*$-algebra $A$ has \emph{real rank zero} if, for every $a=a^*\in A$, and every $\varepsilon>0$, there exists $v=v^*$ that has finite spectrum, such that $||a-v||<\varepsilon$. Finite dimensional $C^*$-algebras have real rank zero. On the same paper, Brown and Pedersen proved that the set of $C^*$-algebras with real rank zero is closed under taking quotients, hereditary $C^*$-subalgebras and inductive limits. So, AF-algebras have real rank zero. A commutative $C^*$-algebra $C(X)$ has real rank zero iff $\dim(X)=0$ iff $C(X)$ is an AF-algebra. In general, there exist real rank zero $C^*$-algebras that are not AF algebras (e.g the irrational rotation algebras). The following basic Lemma is well-known to experts.
\begin{lemma}\label{trace_on_real_rank_zero}
Let $A$ be a unital, infinite dimensional $C^*$-algebra that has real rank zero and $\tau\in T(A)$ be a trace. Then $\inf \{\tau(p), \hspace{5mm} 0\neq p\in  P(A)\}=0.$
\begin{proof}
   For the sake of contradiction, assume that $\inf \{\tau(p), \hspace{5mm} 0\neq p\in P(A)\}=\delta>0.$ By \cite[Lemma 2.2]{scarparo}, there exists a sequence of pairwise orthogonal projections $(p_n)_{n\in \N}$ in $A$. By assumption, $\tau(p_n)\geq \delta$ for every $n$. But this leads to a contradiction, because $\sum_{i=1}^{\infty}p_i \leq 1.$ Proof is complete.
\end{proof}
\end{lemma}
The following Proposition, which is inspired from \cite[Theorem 1.3]{real_rank_zero_an_eigenvalue_variation}, gives us a necessary condition for certain inductive limits (when all the algebras in the sequence map to continuous fields) to be of real rank zero. This will actually be the condition that will fail and cause various group $C^*$-algebras to not have real rank zero.

\begin{proposition}\label{real_rank_zero and f_v}
Let $A$ be a unital $C^*$-algebra that has real rank zero and $A=\varinjlim A_n$, where each $A_n$ is unital.  Denote with $\phi_{mn}:A_n\rightarrow A_m$ and $\mu_n:A_n\rightarrow A$ the connecting maps, which we assume to be unital. Assume that for each $n$, there exists a unital *-homomorphism $\psi_n:A_n\rightarrow B_n$, where $B_n$ is a unital, continuous field over $X_n$, where $X_n$ is a compact, Hausdorff space. Then for every $\varepsilon >0$, $n\geq 1$ and for every $a=a^*$ in $A_n$, there exists $m_0=m_0(\varepsilon, a)\geq n$ such that $\omega(\psi_m(\phi_{mn}(a)))\leq \varepsilon$ for every $m\geq m_0$.
\begin{proof}
We may assume that the self adjoint element $a$ belongs to $A_1$. Let $\varepsilon>0$. Then, because $RR(A)=0$, there exists a self-adjoint element $v$ in $A$ with finite spectrum, let $\{x_1,..,x_r\}$, such that \begin{equation}
    \label{001}
    ||\mu_1(a)-v||<\frac{\varepsilon}{6}.
\end{equation} By the Spectral Theorem,
$$v=\sum_{i=1}^r x_i p_i,$$
where $p_i$ are pairwise orthogonal projections in $A$.
\par
It is known (see for instance \cite[Lemma III.3.1]{davidson_book} that there exists $m_1\in \N$ and $q_1,..,q_r$ pairwise orthogonal projections in $A_{m_1}$ such that 
$$||p_i-\mu_{m_1}(q_i)||<\frac{\varepsilon}{6||v||r}\hspace{5mm} \forall i.$$
For every $m\geq m_1$, set $q_i^{(m)}=\phi_{mm_1}(q_i)$. Then
$$||p_i-\mu_{m}(q_i^{(m)})||<\frac{\varepsilon}{6||v||r}\hspace{5mm} \forall i.$$
Note that $|x_i|\leq ||v||$, so if $v_m=\sum_{i=1}^r x_i q_i^{(m)}$, then
$$||v-\mu_m(v_m)||<\frac{\varepsilon}{6}$$
and $v_m$ has finite spectrum. Thus
$$||\mu_1(a)-\mu_m(v_m)||<\frac{\varepsilon}{3}  $$
for every $m\geq m_1.$
So there exists $m_0$ such that
$$||\phi_{m1}(a)-v_m||\leq \frac{\varepsilon}{2} \hspace{3mm} \text{ for every } m\geq m_0.$$
Hence $||\psi_m(\phi_{m1}(a)-v_m)||\leq \frac{\varepsilon}{2}$. By Corollary \ref{cor_cts_field_3},
$\omega(\psi_m(\phi_{m1}(a)))\leq \varepsilon$, as wished.
\end{proof}

\end{proposition}

\section{Basics on groups }
Let $G$ be a group. We say that a subgroup $H$ is \emph{characteristic in $G$} if for every $\sigma \in Aut(G)$, we have $\sigma(H)\subset H$. Note that this is equivalent to assume $\sigma(H)=H$ for every $\sigma \in Aut(G).$ It is well-known that the commutator subgroup $[G,G] $ and the center $Z(G)$ are characteristic in $ G$. If $G$ is abelian, the torsion subgroup $T(G) $ is characteristic in $ G$. Moreover, if $H $ is characteristic in $ G$ and $G \trianglelefteq L$, then $H \trianglelefteq L$. Also, being a characteristic subgroup is a transitive relation. \par 
Let $G$ be a group. Its \emph{derived series}, indexed $G^{(\alpha)}$, is defined as follows:
\begin{itemize}
    \item $G^{(0)}=G$
    \item $G^{(n+1)}=[G^{(n)},G^{(n)}]$
    \item $G^{(\alpha)}=\bigcap_{\beta<\alpha} G^{(\beta)}$
\end{itemize} 
$G$ is \emph{solvable} if its derived series terminates in finitely many steps. We say that a group is \emph{virtually solvable} if it has a subgroup of finite index that is solvable. Let $G$ be a virtually solvable group. By \cite[p.355]{characteristic_of_finite_index}, it has a characteristic subgroup of finite index that is solvable. Hence, it has a normal series, where each quotient is either abelian or finite. \par

The class of \emph{elementary amenable groups} is the smallest class of groups that contains $\Z$, all finite groups, and it is closed under subgroups, quotients, extensions and increasing unions. Thus all locally virtually solvable groups are elementary amenable. The converse is not true. For more information on elementary amenable groups, we refer the reader to a survey by A.Garrido (\cite{survey_on_EG}).\par 
We say that a group is \emph{periodic} if every element of the group has finite order and \emph{locally finite} if it is an increasing union of finite groups. Every locally finite group is periodic, while the converse holds for a large class of groups, including elementary amenable ones by \cite[Thm 2.3]{elem_amen_groups_chou}. However, the converse does not hold in general. For instance, the Grigorchuk group, which was defined in \cite{Grigorchuk_group} is periodic but not locally finite. Actually, it is an amenable group that is not elementary amenable. \par
The following observation is known to experts, but we present a proof for the sake of completion.

\begin{proposition}\label{form of groups property s}
Let $G$ be a virtually solvable group. Then there is a normal series
$$1=G_0\trianglelefteq G_1\trianglelefteq...\trianglelefteq G_n=G$$ such that
\begin{enumerate}[label=\roman*.]
    \item for every $i$, $G_{i+1}/ G_i$ is either locally finite or abelian.
    \item There is no $i$ such that both $G_{i+1}/ G_i$ and $G_{i+2}/ G_{i+1}$ are locally finite.
\end{enumerate}
\end{proposition}
\begin{proof}

Consider a normal series
$$1=G_0\trianglelefteq G_1\trianglelefteq...\trianglelefteq G_n=G$$
that satisfies $(i)$. and has minimal length (in the sense that there is no normal series of length $m<n$ that satisfies $(i)$). Recall that by the aforementioned, there exists at least one such series. Hence, we can take a minimal one. We will show that such a sequence has to satisfy $(ii)$.
For the sake of contradiction assume that there exists $i$ such that both $G_{i+1}/ G_i$ and $G_{i+2}/ G_{i+1}$ are locally finite. Because the series is normal, $G_i\trianglelefteq G_{i+2}$. Because, both $G_{i+1}/ G_i$ and $G_{i+2}/ G_{i+1}$ are locally finite and the class of locally finite groups is closed under taking extensions, $G_{i+2}/G_i$ is locally finite. So the normal series
$$1=G_0\trianglelefteq G_1 \trianglelefteq..\trianglelefteq G_i\trianglelefteq G_{i+2}\trianglelefteq..\trianglelefteq G_n=G$$ satisfies (i). But also it has length $n-1<n$, contradicting minimality.
\end{proof}

\vspace{3mm}

We end the section with the following Lemma, which we will use in Section 5.
\begin{lemma}\label{surj_abel}
    Let $G$ be a countable, abelian group that is not locally finite. For every non-torsion element $a\in G$, for every $k\in \N$ and for every $d_i\in G$, $\omega(diag(\ref{symbolref}(a),\beta(d_2),...,\beta(d_k)))=2$ in $M_k(C(\widehat{G})).$ 
\begin{proof}
   Let $a\in G$ not torsion and write $f=diag(\beta(a),\beta(d_2),...,\beta(d_k))$. Let $\rho:G\rightarrow G/T(G)$ be the natural surjection. Then $\rho(a)\neq e$. Observe that $\widehat{G/T(G)}$ is connected and $\widehat{G/T(G)}\subset \widehat{G}.$ Because $a$ is non-torsion, $||1-\frac{\rho(a)+\rho(a)^{-1}}{2}||=2$ in $C^*(G/T(G))$.  So, by \cite[Lemma 2.1]{Dadarlat_cts_fields_f.d}, there exists $\chi\in \widehat{G/T(G)}$ such that $||\pi_{\chi}(f)||=2.$ Moreover, if $\iota\in \widehat{G}$ is the trivial character, then $\pi_{\iota}(f)=0$. Thus $\omega(f)=2.$
\end{proof}
\end{lemma}
\section{Basics on group \texorpdfstring{$C^*$}{C*}-algebras}
Let $G$ be a discrete group. By completing the group algebra $\C[G]$ appropriately, we can construct the \emph{full group $C^*$-algebra $C^*(G)$} and the \emph{reduced group $C^*$-algebra $C^*_{r}(G)$}. More specifically, $C_r^*(G):=\overline{\lambda(\C[G])}^{||.||}\subset B(\ell^2(G))$, where $\lambda:\C[G]\rightarrow B(\ell^2(G))$ satisfies $\lambda(g)(\delta_h)=\delta_{gh}$ and is called the \emph{left regular representation}. Furthermore, $C^*(G):=\overline{\C[G]}^{||.||_{\max}}$, where $||x||_{max}=\sup\{||\pi(x)||\hspace{2mm}| \hspace{2mm} \pi:\C[G]\rightarrow B(H) \text{ is a *-representation } \}$. These $C^*$-algebras are isomorphic iff $G$ is amenable. \par 
If $\sigma: G\rightarrow H$ is a group homomorphism, then there exists a unique unital *-homomorphism $\sigma: C^*(G)\rightarrow C^*(H)$ that extends the group homomorphism. The *-homomorphism is injective (surjective) if the group homomorphism is injective (surjective) . Moreover, if $G=\bigcup_{i=1}^{\infty} G_n$, where $(G_n)_{n\geq 1}$ is an increasing sequence of subgroups of $G$, then $C^*(G)=\varinjlim C^*(G_n).$ \par 
For every group $G$, the reduced group $C^*$-algebra admits a faithful trace
$$\tau_G: C^*_r(G)\rightarrow \C$$
satisfying $\tau(a)=<a\delta_e, \delta_e>$, where $\delta_e\in \C[G]\subset C^*_r(G)$ is the function that sends the identity element $e$ to 1 and all the other group elements to zero. For ease of notation, we will identify $x$ with $\delta_x$. \par
If $G$ is a discrete, abelian group, then $C^*(G)\cong C^*_{r}(G)\cong C(\widehat{G})$, where $\widehat{G}=\{\chi:G\rightarrow \TT: \chi(e)=1 \text{ and } \chi(gh)=\chi(g)\chi(h)\}$ is the Pontryagin dual of $G$. We endow $\widehat{G}$ with the topology of pointwise convergence. With this topology, $\widehat{G}$ becomes compact and Hausdorff. If, moreover $G$ is torsion free, then $\widehat{G}$ is connected. Recall that $\widehat{\Z^n}=\TT^n$.\par

Consider the following short exact sequence of groups
\begin{equation}
    \label{111}
    \begin{tikzcd}
\{e\} \arrow{r} & N \arrow{r}{\iota} & G \arrow{r}{\pi} & H \arrow{r} & \{e\}
\end{tikzcd}
\end{equation}

where $|H|=n<\infty$. Let $E:C^*(G)\rightarrow C^*(N)$ be the conditional expectation that satisfies 
$$E(g)=\left\{
\begin{array}{rr}
     & g \text{ if } g\in N\\
     & 0 \text{ if } g\in G \text{ and } g\notin N
\end{array}
\right.$$
For each $h\in H$, fix a lift $g_h\in G$ of $h$.
It can be shown (see \cite[Section 3.2]{Connective} for more details) that there exists a faithful *-homomorphism
$$\Phi:C^*(G)\rightarrow M_n(C^*(N))$$
that satisfies
\begin{equation}
    \label{(2)}
    (\Phi(a))_{h',h}=E(g_{h'}ag_h^{-1}).
\end{equation}

Let $a\in N$. Observe that $\pi(g_{h'}ag_h^{-1})=e$ iff $h=h'.$ Hence
$$\Phi(a)=diag(g_hag_h^{-1}:\hspace{2mm} h\in H).$$
We have $tr_n \otimes\tau_N(\Phi(x))=\tau_G(x)$ for every $x\in C^*(G)$. Indeed, if $x=\sum \lambda_g g\in \C[G]$, then $\tau_G(x)=\lambda_e$. On the other hand, $tr_n \otimes\tau_N(\Phi(x))=\frac{1}{n}\sum_{i=1}^n \tau_N(\Phi(x)_{ii})$. Note that $(\Phi(g))_{ii}=0$ for every $i$, if $g\notin N$. Also, $(\Phi(e))_{ii}=e$ and $\tau_N(g_hxg_h^{-1})=0$ for $e\neq x\in N$, because $e\neq g_hxg_h^{-1}\in N$. Hence, $tr_n \otimes \tau_N(\Phi(x))=\la_e$. Because $\C[G]$ is dense in $C^*(G)$ we have that $tr_n\otimes \tau_N(\Phi(x))=\tau_G(x)$ for every $x\in C^*(G)$. Combining this with the fact that Baum-Connes conjecture holds for torsion free amenable groups and Lemma \ref{trace_on_real_rank_zero}, we deduce the following:
\begin{proposition}\label{torsion_free_by_finite}
Let $G$ be a discrete, infinite, amenable group that has a torsion free subgroup of finite index. Then $C^*(G)$ cannot have real rank zero.
\begin{proof}
    Let $N\trianglelefteq G$ be a normal subgroup of finite index that is torsion free. Because the Baum-Connes conjecture holds for $N$, $Tr_n \otimes \tau_N(p)\in \N$ for every $p\in P_n(C^*(N))$ (see \cite[Prop. 6.3.1, Thm. 1.3]{Baum-Connes_book}). Note that $Tr_n$ is the unnormalized trace on $n\times n$ matrices. By the above, $\tau_G(q)\geq \frac{1}{n}$ for every $0\neq q\in P(C^*(G))$. Result follows from Lemma \ref{trace_on_real_rank_zero}.
\end{proof}
    
\end{proposition}

Now assume that the extension in (\ref{111}) is central and $G$ is amenable (H is no longer assumed to be finite). Then, by \cite[Thm. 1.2]{MR1078249} (see also \cite[Lemma 6.3]{Extension_Cont_field}) $C^*(G)$ is a continuous field over $\widehat{N}$, where $\widehat{N}$ is the Pontryagin dual of $N$. Moreover the fiber on the trivial character of $\widehat{N}$ is isomorphic to $C^*(H)$, while more generally each fiber is isomorphic to some twisted group $C^*$-algebra $C^*(H,\sigma)$. \par 

\section{A class of groups whose  \texorpdfstring{$C^*$}{C*}-algebra is not of real rank zero.}

The following Proposition is a generalization of the ideas in the proof of \cite[Thm. 2.3]{scarparo}.

\begin{proposition}\label{locally_finite_by_nice_is_well_behaved}
Let $$\begin{tikzcd}
 \{e\} \arrow{r} & L \arrow{r}{\iota} & G \arrow{r}{\rho} & H \arrow{r} & \{e\}  
\end{tikzcd}$$
be an extension of groups, where $G$ is amenable and $H$ is locally finite. Assume one of the following holds:
\begin{enumerate}[label=\roman*.]
    \item $Z(L)$ is not locally finite.
    \item $L$ surjects to an abelian group that is not locally finite.
\end{enumerate}
Then $C^*(G)$ does not have real rank zero.
\begin{proof}
  $H$ is locally finite so $H=\bigcup_{n=1}^{\infty} H_n$, where $H_n$ is an increasing sequence of groups and $|H_n|=r_n<\infty$. Set $G_n=\rho^{-1}(H_n)$. Of course, $G_n$ is increasing and $\bigcup_{n=1}^{\infty} G_n=G.$ Moreover, $$\begin{tikzcd}
      \{e\} \arrow{r} & L \arrow{r}{\iota} & G_n \arrow{r}{\rho|_{G_n} } & H_n \arrow{r} & \{e\}
  \end{tikzcd}$$
  is exact. Fix $\{g_h  \hspace{3mm}| \hspace{3mm} h\in H\}$ to be lifts of the elements of $H$ and enumerate such that $h_1=e$ and $H_n=\{h_1,...,h_{r_n}\}$. Let $\Phi_n: C^*(G_n)\rightarrow M_{r_n}(C^*(L))$ be the maps defined in (\ref{(2)}). We have
  $$(\Phi_n(x))_{h',h}=E(g_{h'} x g_h^{-1})$$
 Assume first that (i) holds. Then $M_{r_n}(C^*(L))$ is a continuous field over $\widehat{Z(L)}$.
 Let $a\in Z(L) \leq G_1$ non-torsion. Because $Z(L)$ is a characteristic in $L$, we have that $Z(L) \trianglelefteq G_n$. Hence $\Phi_n(a)=diag(a,d_2,...,d_r)$, where $d_i\in Z(L)$ for every $i$. Because of Lemma \ref{lemma_cts_field_4} and Lemma \ref{surj_abel},
$$\omega_{M_{r_n}C^*(L)}(\Phi_n(\ref{symbolref}(a)))=2 \hspace{3mm} \text{ for every } n\in \N .$$ Result follows from Proposition \ref{real_rank_zero and f_v}. \par 
Now assume that (ii) holds. Then there exists a surjective group homomorphism $\sigma: L\rightarrow N$ for an abelian group $N$ that is not locally finite. Let $b\in N$ not torsion. By surjectivity of $\sigma$, there exists $a\in L\leq G_1$ such that $\sigma(a)=b$. The composition $\sigma \circ \Phi_n: C^*(G_n) \rightarrow M_{r_n}(C^*(N))$ sends 
$a$ to $diag(b,d_2,...,d_r)$, where $d_i\in N$. 
Note that $M_{r_n}(C^*(N))$ is a continuous field over $\widehat{N}$. Because $b$ is not torsion and Lemma \ref{surj_abel},
$$\omega_{M_{r_n}C^*(N)}(\sigma \circ \Phi_n(\ref{symbolref}(a)))=2. \hspace{3mm} \text{ for every } n\in \N .$$ Result follows from Proposition \ref{real_rank_zero and f_v}
\end{proof}
\end{proposition}

By \cite[Chapter I, Lemma 1]{hilman_book}, all infinite, finitely generated elementary amenable groups $G$ satisfy hypothesis (ii) of Proposition \ref{locally_finite_by_nice_is_well_behaved}. Thus, we recover \cite[Thm. 2.3]{scarparo}. \par 
If $G$ is countable, virtually solvable but not locally finite, then Proposition \ref{form of groups property s} implies that $G$ satisfies hypothesis (ii) of Proposition \ref{locally_finite_by_nice_is_well_behaved}. So, we deduce the following Corollary.
\begin{corollary}\label{solvable_not_RR0}
    Let $G$ be a countable, virtually solvable group such that $C^*(G)$ has real rank zero. Then $G$ is locally finite.
\end{corollary}

Now we are ready to prove the main result of the section.
\begin{theorem}\label{thm_1}
Let $$\begin{tikzcd}
    \{e\} \arrow{r} & N \arrow {r} {\iota} & G \arrow{r} & H \arrow{r}{\pi} & \{e\}
\end{tikzcd}$$
be an exact sequence of amenable groups. Assume that $N$ is abelian but not locally finite and $Aut(N)$ is a linear group. Then $C^*(G)$ does not have real rank zero.
\begin{proof}
    Note that there exists a group homomorphism $\sigma: G\rightarrow Aut(N)$ satisfying
    $$\sigma(g)(x)=gxg^{-1} \hspace{5mm} \text{ for every } g\in G, x\in N.$$
    Notice that $\ker(\sigma)=C_G(N)$ and $N \subset Z(C_G(N)).$
Furthermore, we have the short exact sequence of groups
     \begin{equation}
         \label{(4)}
         \begin{tikzcd}
             \{e\} \arrow{r} & C_G(N)\arrow{r} & G \arrow{r}{\sigma} & \sigma(G) \arrow{r} & \{e\}.
         \end{tikzcd}
     \end{equation}
     Because $G$ is amenable, $\sigma(G) \leq Aut(N)$ is amenable. By assumption, $Aut(N)$ is linear, so by the Tits alternative \cite[Thm. 1]{tits_alternative}, $\sigma(G)$ is virtually solvable. Assume, for the sake of contradiction, that $C^*(G)$ has real rank zero. Then $C^*(\sigma(G))$, also has real rank zero. Hence, by Corollary \ref{solvable_not_RR0}, $\sigma(G)$ is locally finite. But we have a contradiction by Proposition \ref{locally_finite_by_nice_is_well_behaved}.
\end{proof}
\end{theorem}
\section{Hirsch length and the main result}
The Hirsch length was first defined for polycyclic groups, and in \cite{hillman_finite_Hirsch_length_old} it was generalized to all elementary amenable groups. More specifically, define $EG_0$ to be the set containing $\Z$ and finite groups. Also define $EG_{\alpha}$ to be $\bigcup_{\beta<\alpha} EG_{\beta}$ if $\alpha$ is a limit ordinal. Otherwise, define it to be the set that arises from the elements of $EG_{\alpha-1}$ after taking increasing unions and extensions. By \cite[Prop. 2.2]{elem_amen_groups_chou}, the set of elementary amenable groups is $\bigcup_{\alpha} EG_{\alpha}$. Define the Hirsch length inductively via the following relations:
\begin{itemize}
\item $h(\Z)=1$
\item $h(G)=0$ for every finite group $G$.
\item $h(H)+h(G/H)=h(G) \hspace{4mm} \text{ for every } H\trianglelefteq G.$
\item $h(\varinjlim G_n)=\sup h(G_n)$ for every increasing net of groups $(G_n)_{n\in \N}.$
\end{itemize}
By \cite[Theorem 1.]{hillman_finite_Hirsch_length_old}, the Hirsch length is well-defined for every elementary amenable group $G$.

Observe that $h(G)=0$ iff $G$ is locally finite. Moreover, $h(\Z^n)=n$ and more generally, in a polycyclic group, $h(G)$ counts the (finite) number of infinite (cyclic) factors. So, virtually polycyclic (and hence finitely generated nilpotent) groups have finite Hirsch length. There exist groups that have finite Hirsch length but are not finitely generated, e.g $h(\Q^n)=n$. There exist groups that have infinite Hirsch length, like $\bigoplus_{\N} \Z$ and $\Z \wr \Z$, where $\wr$ denotes the reduced wreath product. The latter is an example of a finitely generated group that has infinite Hirsch length. \par
We need the following Lemma.
\begin{lemma}\label{autom_abelian_groups}
  Let $G$ be a torsion free abelian group with finite Hirsch length. Then $Aut(G)$ is a linear group.
  \begin{proof}
Because $G$ is torsion free, $Aut(G)\subset Aut(G\otimes_{\Z} \Q).$ Moreover, $G\otimes_{\Z} \Q$ is a $\Q$-vector space. Observe that $G\otimes_{\Z} \Q=\varinjlim G$, where the n-th connecting map is multiplication with $n$. Hence $h(G\otimes_{\Z} \Q)=h(G)<+\infty.$ So, $G\otimes_{\Z} \Q \cong \Q^l$ for $l=h(G)$. Thus $Aut(G)\subset Aut(G\otimes_{\Z} \Q)=GL(l,\Q)$. Proof is complete.
\end{proof}
\end{lemma}

The following Lemma is stated without proof on \cite[p. 238]{Hillman_finite_Hirsch_length}. We will prove it for the sake of completion.
\begin{lemma}\label{unique_max_loc_fin_normal}
   Every group $G$ has a unique maximal normal locally finite subgroup.
   \begin{proof}
       Let $$\E:=\{H\trianglelefteq G \hspace{3mm} | \hspace{3mm} H \text{ is locally finite } \}.$$
       Let $(L_i)_{i\in I}$ be an increasing sequence of locally finite normal subgroups of $G$. Then $L:=\bigcup_{i\in I} L_i$ is a locally finite normal subgroup of $G$. Hence, by Zorn's Lemma, $\E$ has a maximal element. We will show that $\E$ actually has a unique maximal element. For the sake of contradiction, let $\Lambda_1, \Lambda_2$ be distinct maximal elements in $\E$. Because $\Lambda_1, \Lambda_2$ are normal in $G$, the product $\Lambda_1 \Lambda_2$ is a normal subgroup of $G$. It is not difficult to verify that it is also locally finite. Finally, $\Lambda_1 \subsetneq \Lambda_1 \Lambda_2$, contradicting maximality. Proof is complete.
   \end{proof}
\end{lemma}
For the rest of the paper, for a group $G$, we will write $\Lambda(G)$ to denote the unique maximal, normal locally finite subgroup of $G$.
\begin{lemma}\label{characteristic}
For every group $G$, $\Lambda(G)$ is characteristic in $ G.$
\begin{proof}
    Let $\sigma: G\rightarrow G$ be an automorphism of $G$. Then $\sigma(\Lambda(G))$ is a locally finite, normal subgroup of $G$. Hence, by Lemma \ref{unique_max_loc_fin_normal} and its proof, we get that $\sigma(\Lambda(G))\subset \Lambda(G)$. Thus $\Lambda(G) $ is characteristic in $ G.$
\end{proof}
\end{lemma}
Observe that $G/\Lambda(G)$ does not have any non-trivial, normal, locally finite subgroups.
Now we are ready to prove our main result: \par
\vspace{6mm}
\textbf{Proof of Theorem \ref{main-theorem}} \par
We will actually show the contrapositive: i.e we will show that if $G$ is discrete, amenable and has a normal subgroup $H$ which is elementary amenable with finite Hirsch length and is not locally finite, then $C^*(G)$ does not have real rank zero. \par
By \cite[Theorem]{Hillman_finite_Hirsch_length}, we have that $H/ \Lambda(H)$ is virtually solvable. Because $\Lambda(H) $ is characteristic in $ H$, we have that $\Lambda (H) \trianglelefteq G$ and $H/ \Lambda(H) \trianglelefteq G/ \Lambda(H).$ Note that amenability and finite Hirsch length is preserved under taking quotients. Moreover the property of having real rank zero is preserved under taking quotients. Hence, we may assume that $H$ is non-trivial, is virtually solvable with finite Hirsch length and has no non-trivial normal locally finite subgroups. \par 
By \cite[p.355]{characteristic_of_finite_index}, there exists $L$ characteristic in $ H$ that is solvable and it has finite index in $H$. Note that $L \trianglelefteq G$ so we may assume that $H$ is non-trivial, is solvable with finite Hirsch length and has no normal, non-trivial, locally finite subgroups. \par
Because $H$ is solvable, its derived series has finite length. Because the commutator of a group is a characteristic subgroup, the last non-zero group on the derived series, let $S$ is an abelian characteristic subgroup of $H$. Moreover, by assumption, the fact that in an abelian group the torsion subgroup is characteristic, and the fact that being a characteristic subgroup is a transitive property, we deduce that $S$ is torsion free. Otherwise the torsion subgroup would be a non-zero normal locally finite subgroup of $H$. Note that $S\trianglelefteq G$, so we may assume that $H$ is non-trivial, is abelian, torsion free with finite Hirsch length. But now result yields from Lemma \ref{autom_abelian_groups} and Theorem \ref{thm_1}. \qed
\par
\vspace{5mm}
We will end the section with one more Corollary of Theorem \ref{main-theorem}. \par

We say that a group $G$ is with \emph{infinite conjugacy classes}, or \emph{icc} if it is not trivial and for every $a\neq e$, the set $\{ g^{-1}ag \hspace{3mm} | \hspace{3mm} g\in G\}$ is infinite. Examples of icc groups are $S_{\infty}$ and all free groups. On the other hand, nilpotent groups cannot be icc as they have non-trivial center.
\begin{corollary}
    Let $G$ be a discrete amenable group such that $C^*(G)$ has real rank zero. 
    There exists a locally finite, normal subgroup $H\trianglelefteq G$ such that $G/H$ is either trivial or icc.
    \begin{proof}
      Assume that $G$ is not locally finite. We will show that $\{e\}\neq G/\Lambda(G)$ is icc, which is enough to yield the result. For the sake of contradiction, assume that $G/\Lambda(G)$ is not icc. By construction, it has no non-trivial, locally finite, normal subgroups. So, by \cite[Proposition 1.1]{icc_extension}, $G/\Lambda(G)$ has a normal subgroup isomorphic to $\Z^n$ for some $n$. By assumption $C^*(G/\Lambda(G))$ has real rank zero. Thus, we have a contradiction by Theorem \ref{main-theorem}.
    \end{proof}
\end{corollary}
\section{Groups where real part of all non-torsion elements is far away from elements of finite spectrum in the group \texorpdfstring{$C^*$}{C*}-algebra}
Our main difficulty in proving Conjecture \ref{conjecture}, at least in the case of elementary amenable groups, arises from the fact that we can find increasing sequences of $C^*$-algebras with real rank greater than zero (e.g matrices over $C(\TT)$), such that the inductive limit has real rank zero). That's why define a property for groups, whose presence implies that the full group $C^*$-algebra does not have real rank zero. A crucial feature of this property is that it is preserved when taking increasing unions (Prop. \ref{excellent_ind_limits}). \par
Set $d(a):=\inf\{ ||a-v||, \hspace{5mm} v=v^*\in C^*(G) \text { and }v \text{ has finite spectrum}\}.$ Note that $d(a)\leq ||a||$ and $d(a\pm \lambda\cdot 1)=d(a)$ for every $\lambda\in \C$. A $C^*$-algebra $A$ has real rank zero, iff $d(a)=0$ for every $a=a^*\in A$.
\begin{definition}\label{excellent_group}
We say that a group $G$ is \emph{strongly not (FS)} if it is not periodic and for every $g\in G$ that is non-torsion, we have that $d(\frac{g+g^{-1}}{2})=1.$
\end{definition}
Obviously, if $G$ is strongly not (FS), $C^*(G)$ cannot have real rank zero. We do not know if the converse holds. If it holds, then \cite[Thm 2.3]{scarparo} would automatically yield that Conjecture \ref{conjecture} would be true for all elementary amenable groups. Also Proposition \ref{locally_finite_by_nice_is_well_behaved} and Theorem \ref{main-theorem} would yield that the group $C^*$-algebras of more (non-elementary) amenable groups do not have real rank zero.
\begin{proposition}\label{excellent_ind_limits}
If $G_n$ is an increasing union of strongly not (FS) groups, then $G=\bigcup_{i=1}^{\infty} G_n$ is also strongly not (FS).
  \begin{proof}
      Let $g\in G$ non-torsion, $\varepsilon>0$ and $w=w^*\in C^*(G)$ that has finite spectrum. By the spectral theorem, there is $n\in \N$ such that $g\in G_n$ and $v=v^*\in C^*(G_n)$ with finite spectrum such that $||w-v||<\varepsilon.$ Because $G_n$ is strongly not (FS), $||\frac{g+g^{-1}}{2}-v||\geq 1.$ Thus  $||\frac{g+g^{-1}}{2}-w||\geq 1-\varepsilon.$ Because $\varepsilon>0$, $g$ and $w$ were selected at random, $d(\frac{g+g^{-1}}{2})=1$ for every $g\in G$ non-torsion. Hence $G$ is strongly not (FS).
  \end{proof}  
\end{proposition}
By Lemma \ref{surj_abel} and Lemma \ref{lemma_cts_field_1}, every countable abelian group that is not locally finite is strongly not (FS).

\begin{lemma}\label{locally_finite_by_torsion_free_excellent}
Let $$\begin{tikzcd}
    \{e\} \arrow{r} & N \arrow{r}{\iota} & G \arrow{r}{\pi} & H \arrow{r} & \{e\}  
\end{tikzcd}$$
be short exact sequence of groups, with $N$ periodic and $H\neq \{e\}$ torsion free amenable. Then $G$ is strongly not (FS).
\begin{proof}
    Let $g\in G$ non-torsion. Because $N$ is periodic, $g\notin G$. Hence $\pi(g)\neq e$. Because $H$ is torsion free amenable, the Kadison-Kaplansky conjecture holds for $H$ by \cite[Thm. 1.3]{Baum-Connes_book}. Hence $C^*(H)$ does not have any non-trivial projections. Hence, we have $d(\frac{g+g^{-1}}{2})\geq d(\frac{\pi(g)+\pi(g)^{-1}}{2})=1.$ Thus $d(\frac{g+g^{-1}}{2})=1$. Because $g$ is a random non-torsion element, we deduce that $G$ is strongly not (FS).
\end{proof}
\end{lemma}

\begin{lemma}\label{nilpotent_is_loc_fin_by_TF}
Every nilpotent group is periodic by torsion free.
\begin{proof}
    Let $G$ be a nilpotent group and set $$N:=\{x\in G \hspace{3mm} | \hspace{3mm} o(g)<\infty\}.$$
    By \cite[Ch. 1, Cor. 10]{Segal_polycyclic}, $N$ is a periodic, normal subgroup of $G$ and $G/N$ is torsion free. Result follows.
\end{proof}

\end{lemma}

Combining Proposition \ref{excellent_ind_limits}, Lemma \ref{locally_finite_by_torsion_free_excellent}, Lemma \ref{nilpotent_is_loc_fin_by_TF} and the fact that an elementary amenable, periodic group is locally finite, we deduce the following:

\begin{proposition}\label{locally_nilpotent_is_excellent}
Let $G$ be a locally nilpotent group that is not locally finite. Then it is strongly not (FS). In particular, $C^*(G)$ does not have real rank zero.

\end{proposition}
The following example from \cite{hillman_finite_Hirsch_length_old} yields a new example not covered by Theorem \ref{thm_1}. 
\begin{example} 
Let $x_k\in \bigoplus_{\Z} \Z$ be the sequence that has 1 on the $k$-th position and 0 elsewhere. For $i\in \Z$, consider the automorphism $e_i \in Aut(\bigoplus_{\Z} \Z)$ that satisfies $e_i(x_i)=x_i+x_{i+1}$ and $e_i(x_j)=x_j$ if $i\neq j.$ Let $G\leq Aut(\bigoplus_{\Z} \Z)$ be the group generated by $\{e_i\in \Z\}$. As $G$ is an increasing union of subgroups isomorphic to groups of upper triangular matrices, it is locally nilpotent. So, by Proposition \ref{locally_nilpotent_is_excellent}, $C^*(G)$ does not have real rank zero. On the other hand, it can be shown that $G$ has no non-trivial normal, abelian subgroups and it has infinite Hirsch length.

\end{example}
\vspace{5mm}
\textbf{Acknowledgements:} I would like to thank my advisor Marius Dadarlat for introducing me to the problem and for making useful comments on earlier drafts of this paper. I would also like to thank the referee for their useful comments.

\vspace{10mm}

\begin{center}

\bibliographystyle{abbrv}
\small

\bibliography{paper_2}
\end{center}
\emph{Iason Moutzouris: Department of Mathematics, Purdue University, West Lafayette, IN 47907, USA }\par

\emph{Email address: imoutzou@purdue.edu} \par
\emph{Website: https://sites.google.com/view/iasonmoutzouris/}
\end{document}